\documentclass[10pt
]{article}

\RequirePackage[OT1]{fontenc}
\RequirePackage[
amsthm,amsmath
]{imsart}

\usepackage{graphicx}
\usepackage{latexsym,amsmath}
\usepackage{amsmath,amsthm,amscd}
\usepackage{amsfonts}
\usepackage[psamsfonts]{amssymb}

\usepackage{mathabx} 

\usepackage{enumerate}

\usepackage{url}

\usepackage{tocvsec2}

\usepackage{epstopdf}
\usepackage{wrapfig}
\usepackage{float}

\usepackage{calrsfs}

\usepackage[small]{caption}
\usepackage{hyperref}

\startlocaldefs

\theoremstyle{plain} 
\newtheorem{theorem}{Theorem}[section]
\newtheorem{corollary}[theorem]{Corollary}

\newtheorem{lemma}[theorem]{Lemma}

\newtheorem{proposition}[theorem]{Proposition}

\theoremstyle{definition} 

\theoremstyle{definition} 

\newtheorem*{ex*}{Example}
\theoremstyle{remark} 

\theoremstyle{remark} 

\newtheorem*{remark*}{Remark}

\numberwithin{equation}{section}

\makeatletter
\def\subsubsubsection{\@startsection{subsubsubsection}{4}{\z@}{-3.25ex plus -1ex minus -.2ex}{1.5ex plus .2ex}{\normalsize}}
\makeatother                                   
 
\newcommand{\beqa}{\begin{eqnarray}}
\newcommand{\eeqa}{\end{eqnarray}}
 
\newcommand{\bseq}{\begin{subequations}}
 
\newcommand{\eseq}{\end{subequations}}

\newcommand{\dd}{\partial}

\newcommand{\dom}{{\,\operatorname{dom}}}

\renewcommand{\dd}{{\,\operatorname{d}}}

\newcommand{\supp}{\operatorname{supp}}

\newcommand{\al}{\alpha}

\newcommand{\Si}{\Sigma}

\newcommand{\be}{\beta}

\renewcommand{\Psi}{\overline{\Phi}}

\newcommand{\II}{\mathcal{I}}

\renewcommand{\SS}{\mathcal{S}}

\renewcommand{\nu}{\mathsf{nu}}

\newcommand{\ii}{\operatorname{I}}

\renewcommand{\P}{\operatorname{\mathsf{P}}} 
 
\newcommand{\E}{\operatorname{\mathsf{E}}}

\newcommand{\cc}{{\operatorname{\mathsf{c}}}}

\newcommand{\R}{\mathbb{R}}

\newcommand{\tE}{{\tilde{E}}}

\renewcommand{\tt}{{\tilde{t}}}

\renewcommand{\tt}{{\mathbf{t}}}

\renewcommand{\le}{\leqslant}
\renewcommand{\ge}{\geqslant}

\newcommand{\li}[1]{{{#1}^*}^{-1}} 
\newcommand{\tli}[1]{\widetilde{{{#1}^*}^{-1}}} 

  \newcommand{\fJ}{\operatorname{\raisebox{.8pt}{\fbox{\scriptsize \mathsf{H}}}}}
 \renewcommand{\fJ}{\mathrel{\raisebox{.0pt}{\fbox{\scriptsize H}}}}
 \renewcommand{\fJ}{\mathrel{\raisebox{.0pt}{\fbox{\scriptsize $\operatorname{H}$}}}}
 \renewcommand{\fJ}{\mathrel{\raisebox{.0pt}{\fbox{\scriptsize $\operatorname{\mathsf{H}}$}}}}
   \newcommand{\fJt}{\operatorname{\raisebox{.8pt}{\fbox{\tiny H}}}}
    \renewcommand{\fJt}{\mathrel{\raisebox{.0pt}{\fbox{\tiny $\operatorname{\mathsf{H}}$}}}}

\endlocaldefs

\begin{document}

\begin{frontmatter}

\title{
(Quasi)additivity properties of the Legendre--Fenchel transform and its inverse, with applications in probability}
\runtitle{Inverse of the Legendre--Fenchel transform}

%

\begin{aug}
\author{\fnms{Iosif} \snm{Pinelis}\thanksref{t2}\ead[label=e1]{ipinelis@mtu.edu}}
  \thankstext{t2}{Supported by NSA grant H98230-12-1-0237}
\runauthor{Iosif Pinelis}


\address{Department of Mathematical Sciences\\
Michigan Technological University\\
Houghton, Michigan 49931, USA\\
E-mail: \printead[ipinelis@mtu.edu]{e1}}
\end{aug}

\begin{abstract}
The notion of the H\"older convolution is introduced. The main result is that, under general conditions on functions $L_1,\dots,L_n$, 
one has $\li{(L_1\fJt\cdots\fJt L_n)}=\li{L_1}+\dots+\li{L_n}$, where $\fJt$ denotes the H\"older convolution and $\li 
L$ is the function inverse to the Legendre--Fenchel transform $L^*$ of a given function
$L$
. General properties of the functions $L^*$ and $\li L$ are discussed. Applications to probability theory are presented. 
In particular, an upper bound on the quantiles of the distribution of the sum of random variables is given. 
\end{abstract}

  
%

\setattribute{keyword}{AMS}{AMS 2010 subject classifications:}





\end{frontmatter}

\settocdepth{chapter}

\tableofcontents 

\settocdepth{subsubsection}

\theoremstyle{plain} 





\section{
(Quasi)additivity properties of the Legendre--Fenchel transform and its inverse}

For brevity, let $T:=(0,\infty)$. Take any function $L\colon T\to[-\infty,\infty]$. 
To avoid unpleasant trivialities, assume that 
\begin{equation}\label{eq:ne infty}
	L(T)\ne\{\infty\}; 
\end{equation}
that is, for some $t\in T$ one has $L(t)<\infty$. 

The Legendre--Fenchel transform $L^*$ of $L$ may be defined by the formula 
\begin{equation}\label{eq:L^*:=}
	L^*(x):=\sup_{t\in T}[tx-L(t)]
\end{equation}
for all $x\in\R$, so that $L^*(x)$ may take any value on the extended real line $[-\infty,\infty]$.

Next, introduce the function $\li L$, which is the 
generalized inverse of the Legendre--Fenchel transform $L^*$, by the formula 
\begin{equation}\label{eq:li L}
	\li L(u):=\inf E_L(u),\quad\text{where}\quad E_L(u):=\{x\in\R\colon L^*(x)\ge u\}, 
\end{equation}
for all $u\in\R$ (recall that, according to the standard convention, for any subset $E$ of $\R$, $\inf E=\infty$ if and only if $E=\emptyset$). 

Take now any natural $n$ and any 
functions $L_1,\dots,L_n$ mapping $T$ into $(-\infty,\infty]$. 
Then introduce what we shall refer to as the \emph{H\"older convolution}, $L_1\fJ\cdots\fJ L_n$, by the formula 
\begin{equation*}
	(L_1\fJ\cdots\fJ L_n)(t)
	:=\inf\Big\{
	\sum_{j=1}^n\al_j L_j\Big(\frac t{\al_j}\Big)\colon(\al_1,\dots,\al_n)\in\Si_n
	\Big\} 
\end{equation*}
for all $t\in\R$, where 
\begin{equation*}
	\Si_n
	:=\big\{\textstyle{
	(\al_1,\dots,\al_n)\in(0,\infty)^n\colon \sum_{j=1}^n\al_j=1
	}
	\big\}. 
\end{equation*}
%
%
The reason for using the term ``H\"older convolution'' will be apparent later. 

At this point, let us just note the following additivity property of the H\"older convolution with respect to the family of power functions (cf.\ \cite{maligr}):  
\begin{equation}\label{eq:powers}
	p_{r,a}\fJ p_{r,b}=p_{r,a+b}
\end{equation}
for all real $r\ge1$ and $a,b\ge0$, where $p_{r,a}(t):=(at)^r$ for all $t\in T$.   
An immediate application of \eqref{eq:powers} is the following proof of the Minkowski inequality, say for any random variables (r.v.'s) $X$ and $Y$ and any real $r\ge1$ (cf.\ \cite{maligr}): 
\begin{align*}
	\|X+Y\|_r^r\le
	\E(|X|+|Y|)^r  
	&=\E p_{r,|X|+|Y|}(1) \\ 
	&=\E(p_{r,|X|}\fJ p_{r,|Y|})(1) \\ 
	&=\E\inf_{0<\be<1}\big[\be\, p_{r,|X|}\big(\tfrac1\be\big)+(1-\be)\,p_{r,|Y|}\big(\tfrac1{1-\be}\big)\big] \\ 
&\le\inf_{0<\be<1}\E\big[\be\, p_{r,|X|}\big(\tfrac1\be\big)+(1-\be)\,p_{r,|Y|}\big(\tfrac1{1-\be}\big)\big] \\
&=\inf_{0<\be<1}\big[\be\, p_{r,\|X\|_r}\big(\tfrac1\be\big)+(1-\be)\,p_{r,\|Y\|_r}\big(\tfrac1{1-\be}\big)\big] \\ 
&=(p_{r,\|X\|_r}\fJ p_{r,\|Y\|_r})(1) 
=p_{r,\|X\|_r+\|Y\|_r}(1) 
=(\|X\|_r+\|Y\|_r)^r. 
\end{align*}

The following theorem, which expresses an additivity property of the inverse of the Legendre--Fenchel transform, is the main result in this paper. 
\begin{theorem}\label{th:1}
Suppose that the condition \eqref{eq:ne infty} holds with $L_j$ in place of $L$, for each $j\in\{1,\dots,n\}$. Then 
\begin{equation}\label{eq:}
	\li{(L_1\fJ\cdots\fJ L_n)}=\li{L_1}+\dots+\li{L_n}. 
\end{equation}
\end{theorem}

Note that the right-hand of inequality is correctly defined, in view of property \eqref{li<infty} in Proposition~\ref{prop:gen}. 

Theorem~\ref{th:1} is based on the following proposition, which appears to be of independent interest as well. 

\begin{proposition}\label{prop:L^*} 
For any real $x_1,\dots,x_n$
\begin{equation}\label{eq:L^*}
	\min\big(L_1^*(x_1),\dots,L_n^*(x_n)\big)
	\le(L_1\fJ\cdots\fJ L_n)^*(x_1+\dots+x_n)
	\le\max\big(L_1^*(x_1),\dots,L_n^*(x_n)\big). 
\end{equation}
\end{proposition} 

To quickly appreciate the relevance of Proposition~\ref{prop:L^*} regarding Theorem~\ref{th:1}, one can first make the easy observation that the function $L^*$ is always nondecreasing (cf.\ Proposition~\ref{prop:gen}\eqref{L^* gen}). 
Consider now the easier case when the functions $L_1^*,\dots,L_n^*$ and $(L_1\fJ\cdots\fJ L_n)^*$ are continuous and strictly increasing. Suppose then that for some real $u$ and $x_1,\dots,x_n$ and for all $j\in\{1,\dots,n\}$ one has $L_j^*(x_j)=u$, so that $x_j=\li L(u)$. Then Proposition~\ref{prop:L^*} yields $(L_1\fJ\cdots\fJ L_n)^*(x_1+\dots+x_n)=u$ and hence 
$\li{(L_1\fJ\cdots\fJ L_n)}(u)=x_1+\dots+x_n=\sum_1^n\li{L_j}(u)$; cf.\ \eqref{eq:}. 
Of course, here there will be some technical difficulties to overcome, since in general we do not assume the additional conditions that the functions $L_1^*,\dots,L_n^*$ and $(L_1\fJ\cdots\fJ L_n)^*$ are continuous and strictly increasing and all the equations $L_j^*(x_j)=u$ for $x_j$ have solutions for all real $u$. 
However, as shown in the proof of Theorem~\ref{th:1} below, these difficulties are not excessive. 

\begin{proof}[Proof of Proposition~\ref{prop:L^*}] 
Take indeed any real $x_1,\dots,x_n$. 

Consider the bijective correspondence 
\begin{equation}\label{eq:1-to-1} 
	(0,\infty)\times\Si_n\ni(t,\al_1,\dots,\al_n)\longleftrightarrow
	\tt:=(t_1,\dots,t_n):=\big(\tfrac t{\al_1},\dots\tfrac t{\al_n}\big)\in(0,\infty)^n,  
\end{equation}
under which one has  
$\al_j=\al_j(\tt):=\frac{1/t_j}{1/t_1+\dots+1/t_n}$ for all $j\in\{1,\dots,n\}$ and $t=\frac1{1/t_1+\dots+1/t_n}$.  

It follows that for 
\begin{equation}\label{eq:L:=}
	L:=L_1\fJ\cdots\fJ L_n  
\end{equation} 
one has 
\begin{align*}
	L^*\big(\textstyle{\sum_1^n}x_j\big)&=\sup_{t>0}\big[t\,\textstyle{\sum_1^n}x_j-L(t)\big] \\ 
	&=\sup\big\{t\,\textstyle{\sum_1^n}x_j-\textstyle{\sum_1^n}\al_j L_j\big(\tfrac t{\al_j}\big)\colon
	(t,\al_1,\dots,\al_n)\in(0,\infty)\times\Si_n\big\} \\ 
	&=\sup\big\{\textstyle{\sum_1^n}
	\al_j(\tt)\,[t_jx_j-L_j(t_j)]\colon 
	(t_1,\dots,t_n)\in(0,\infty)^n\big\} \\
	&\le\sup\big\{\max_{1\le j\le n}
	[t_jx_j-L_j(t_j)]\colon 
	(t_1,\dots,t_n)\in(0,\infty)^n\big\} \\
	&=\max_{1\le j\le n}\sup_{t_j>0}[t_jx_j-L_j(t_j)]
	=\max_{1\le j\le n}L_j^*(x_j),  	
\end{align*}
which proves the second inequality in \eqref{eq:L^*}. 
Somewhat similarly, the expression in the third line of the above multi-line display is no less than 
\begin{equation}\label{eq:last}
\sup\big\{\min_{1\le j\le n}
	[t_jx_j-L_j(t_j)]\colon 
	(t_1,\dots,t_n)\in(0,\infty)^n\big\}\ge\min_{1\le j\le n}\sup_{t_j>0}[t_jx_j-L_j(t_j)]
=\min_{1\le j\le n}L_j^*(x_j).  	
\end{equation}
To verify the inequality in \eqref{eq:last} (which is in fact an equality), for each $j\in\{1,\dots,n\}$ take any real $v_j<\sup_{t_j>0}[t_jx_j-L_j(t_j)]$ and then take any real $t_j>0$ such that $t_jx_j-L_j(t_j)>v_j$, whence the supremum in \eqref{eq:last} is greater than $\min_{1\le j\le n}v_j$, for any real numbers $v_j$ less than $\sup_{t_j>0}[t_jx_j-L_j(t_j)]$. Thus, the first inequality in \eqref{eq:L^*} is proved as well. 
%
%
%
%
\end{proof}

Now one is ready to complete

\begin{proof}[Proof of Theorem~\ref{th:1}] 
Take any $u\in\R$. 
Recall the definition of the set $E_L(u)$ in \eqref{eq:li L} and let, for brevity,
\begin{equation}\label{eq:z_j}
	E:=E_L(u),\quad z:=\inf E, \quad E_j:=E_{L_j}(u),\quad z_j:=\inf E_j, 
\end{equation}
where $L$ is as in \eqref{eq:L:=} and $j\in\{1,\dots,n\}$. 
Since the functions $L^*,L^*_1,\dots,L^*_n$ are nondecreasing, each of the sets $E,E_1,\dots,E_n$ is an interval in $\R$ whose right endpoint is $\infty$, and the complements $E^\cc,E_1^\cc,\dots,E_n^\cc$ of these sets are intervals whose left endpoint is $-\infty$. 

On the other hand, by Proposition~\ref{prop:L^*}, the Minkowski sum $E_1+\dots+E_n$ of the sets $E_1,\dots,E_n$ is a subset of $E$, which yields $z_1+\dots+z_n\ge z$. 
Quite similarly, $E_1^\cc+\dots+E_n^\cc\subseteq E^\cc$, which yields $z_1+\dots+z_n\le z$, since $\sup E_j^\cc=\inf E_j=z_j$ for all $j$ and $\sup E^\cc=\inf E=z$. 
We conclude that $z_1+\dots+z_n=z$. 
It remains to recall the definitions in \eqref{eq:z_j} and \eqref{eq:li L}. 
\end{proof}



\begin{corollary}\label{th:2}
In the conditions of Theorem~\ref{th:1}, 
suppose also that, for each $j\in\{1,\dots,n\}$, $L_j(t)/t$ is nondecreasing in $t>0$    
\big(in particular, this will be the case when $L_j$ is convex with $L_j(0+)\le0$\big). 
Then 
\begin{equation*}
	\li{(L_1+\dots+L_n)}\le\li{L_1}+\dots+\li{L_n}. 
\end{equation*}
\end{corollary}

This follows immediately from Theorem~\ref{th:1}; indeed, if $L_j(t)/t$ is nondecreasing in $t>0$ for each $j\in\{1,\dots,n\}$, then $L_1+\dots+L_n\le L_1\fJ\dots\fJ L_n$. 


Instead of the definition \eqref{eq:li L} of the (smallest possible) version, $\li L$, of the generalized inverse of the $L^*$, one can consider the following largest possible version of it, given 
by the formula 
\begin{equation}\label{eq:tli L}
	\tli L(u):=\inf\{x\in\R\colon L^*(x)>u\}
\end{equation}
for all $u\in\R$. 
The functions $\li L$ and $\tli L$ are closely related: 

\begin{proposition}\label{prop:li,tli} 
For any $u\in\R$
\begin{equation*}
	\tli L(u)=\li L(u+)\quad\text{and}\quad\li L(u)=\tli L(u-). 
\end{equation*}
\end{proposition} 

\begin{proof}[Proof of Proposition~\ref{prop:li,tli}]
Introduce the sets 
\begin{equation*}
	E_u:=\{x\in\R\colon L^*(x)\ge u\}\quad\text{and}\quad\tE_u:=\{x\in\R\colon L^*(x)>u\}; 
\end{equation*}
here and subsequently in this proof, $u$ stands for an arbitrary real number. 
Then 
$\tE_u=\bigcup_{v>u}E_v$ and $E_u=\bigcap_{w<u}\tE_w$. Hence, 
\begin{equation}\label{eq:li,tli}
	\tli L(u)=\inf\tE_u=\inf_{v>u}\inf E_v
	=\inf_{v>u}\li L(v)=\li L(u+); 
\end{equation}
the second equality here follows because $\inf\bigcup\limits_{S\in\SS}S=\inf\limits_{S\in\SS}\inf S$ for any set $\SS$ of subsets of $\R$, whereas the last equality in \eqref{eq:li,tli} is due to the function $\li L$ being nondecreasing \big(see Proposition~\ref{prop:gen}\eqref{li L}\big). 
Somewhat similarly, 
\begin{equation*}
	\li L(u)=\inf E_u=\sup_{w<u}\inf\tE_w
	=\sup_{w<u}\tli L(w)=\tli L(u-); 
\end{equation*}
the second equality here follows because, by Proposition~\ref{prop:gen}\eqref{L^* gen}, the sets $\tE_w$ are intervals in $\R$ with the right endpoint equal $\infty$, and $\inf\bigcap\limits_{I\in\II}I=\sup\limits_{I\in\II}\inf I$ for any set $\II$ of such intervals. 
\end{proof}

It follows by Proposition~\ref{prop:li,tli} and Proposition~\ref{prop:gen}\eqref{li<infty} that Theorem~\ref{th:1} and Corollary~\ref{th:2} 
hold with $\tli{L_j}$ in place of $\li L_j$. It further follows that these results will hold for any ``weighted'' generalized inverses of the form $(1-\al)\li L_j+\al\tli{L_j}$, for any fixed $\al\in[0,1]$. 
However, in applications in probability as the one to be presented in Corollary~\ref{cor:}, one should prefer to use $\li L_j$, the smallest member of this family of generalized inverses -- because, at least formally, this choice maximizes the left-hand side of inequality \eqref{eq:cor} and thus provides the strongest version of that inequality. 

On the other hand, it is straightforward to modify the proof of Theorem~\ref{th:1} so that to obtain its counterpart for the largest generalized inverse directly, rather than via Proposition~\ref{prop:li,tli}. 
Also, Proposition~\ref{prop:li,tli} allows one to obtain an alternative proof of Theorem~\ref{th:1} (as originally stated, for the smallest generalized inverse), based on the following proposition (for the largest generalized inverse). 

\begin{proposition}\label{prop:li expl} 
For any $u\in\R$
\begin{equation}\label{eq:tli=}
	\tli L(u)=\hat L(u):=\inf_{t>0}\frac{u+L(t)}t. 
\end{equation}
\end{proposition} 

Indeed, for $L$ as in \eqref{eq:L:=} and 
for any $u\in\R$, twice using Proposition~\ref{prop:li expl} \big(and also the bijective correspondence in \eqref{eq:1-to-1}\big) one has 
\begin{equation}\label{eq:rio12}
\begin{aligned}
\tli L(u)
=\hat L(u)
=\inf_{t>0}\tfrac{u+L(t)}t
&=\inf\big\{
       \tfrac1t\,\big[u+\textstyle{\sum_1^n}\al_j L_j\big(\tfrac t{\al_j}\big)\big]
       \colon(t,\al_1,\dots,\al_n)\in(0,\infty)\times\Si_n\} \\ 
&=\inf\big\{
       \textstyle{\sum_1^n}\tfrac{u+L_j(t_j)}{t_j}
       \colon(t_1,\dots,t_n)\in(0,\infty)^n\} \\ 
&=\textstyle{\sum_1^n}\inf_{t_j>0}\frac{u+L_j(t_j)}{t_j}
=\textstyle{\sum_1^n}\tli{L_j}(u),   
\end{aligned}	
\end{equation}
so that one has \eqref{eq:} with $\tli{L_j}$ instead of $\li{L_j}$. 


Let us now present  

\begin{proof}[Proof of Proposition~\ref{prop:li expl}] 
For any real $u$ and $x$  
\begin{equation*}
	L^*(x)>u\iff\big(\exists t>0\ \,tx-L(t)>u\big)\iff\big(\exists t>0\ \,x>\tfrac{u+L(t)}t\big)\iff x>\inf_{t>0}\tfrac{u+L(t)}t;  
\end{equation*}
thus, $L^*(x)>u\iff x>\hat L(u)$. It remains to recall the definition \eqref{eq:tli L} of $\tli L$. 
\end{proof}


Identity \eqref{eq:tli=} was given by Rio \cite[page~159]{rio00} and \cite[(4)]{rio12} in the case when the function $L$ is convex, nondecreasing, left-continuous, and with $L(0)=0$; also, it was tacitly assumed in the proof there that the infimum in \eqref{eq:tli=} is attained and the function $L^*$ is strictly increasing. 

\section{Applications to probability} \label{prob} 

Take any real-valued random variable (r.v.) $X$ and consider the corresponding Cram\'er--Chernoff function $L_X$ defined by the formula 
\begin{equation}\label{eq:L_X}
	L_X(t):=\ln\E e^{tX}
\end{equation}
for all $t\in(0,\infty)$; thus, $L_X$ is the logarithm of the moment generating function of $X$. 
Note that $L_X(t)\in(-\infty,\infty]$ for all $t\in(0,\infty)$. 
Note also that the condition \eqref{eq:ne infty} is satisfied with $L_X$ in place of $L$ if and only if 
\begin{equation}\label{eq:<infty}
	\E e^{tX}<\infty \text{ for some $t\in T$,} 
\end{equation}
which will be henceforth assumed; then, by previous discussion, $L_X^*(x)>-\infty$ and $\li L_X(u)<\infty$ for all real $x$ and $u$.  
 
By Markov's inequality, 
\begin{equation}\label{eq:tail<}
	\P(X\ge x)\le\exp\{-L_X^*(x)\}  
\end{equation}
for all $x\in\R$, with the convention $\exp\{-\infty\}:=0$. 

Take now any real $u$ and then any $z\in(y,\infty)$, where $y:=\li L(u)$ and $L:=L_X$. 
Then 
there is $x_z\in(-\infty,z)$ such that $L^*(x_z)\ge u$. 
Since $L^*$ is nondecreasing, it follows that $L^*(z)\ge u$, for all $z\in(y,\infty)$. 
So, 
$
	\P(X>y)=\lim_{z\downarrow y}\P(X>z)
	\le\limsup_{z\downarrow y}\exp\{-L^*(z)\}\le e^{-u}. 
$ 
Therefore, 
\begin{equation}\label{eq:>,<e^-u}
	\P\!\big(X>\li{L_X}(u)\big)\le e^{-u} 
\end{equation}
for all real $u$. 
Thus, $\li{L_X}\big(-\ln(1-q)\big)$ may be considered as an upper bound on the $q$-quantile of the distribution of $X$, for any $q\in(0,1)$. 

Comparing \eqref{eq:tail<} and \eqref{eq:>,<e^-u}, one may ask as to whether 
one can write 
\begin{equation}\label{eq:ge,<e^-u}
	\P\!\big(X\ge\li{L_X}(u)\big)\le e^{-u}  
\end{equation} 
for all real $u$. 
A complete answer to this question is given by


\begin{proposition}\label{prop:} 
Take any $u\in\R$. 
Then inequality \eqref{eq:ge,<e^-u} fails to hold if and only if all of the following conditions take place: 
\begin{enumerate}[(i)]
	\item \label{x max<infty} $x_{\max}:=\sup\supp X<\infty$ (here, as usual, $\supp X$ denotes the support set of the distribution of $X$); 
	\item \label{p max>0} $p_{\max}:=\P(X=x_{\max})>0$; 
	\item \label{u ge} $u>-\ln p_{\max}$.  
\end{enumerate}
\end{proposition}

A proof of Proposition~\ref{prop:} is given at the end of Appendix~\ref{appendix}. 

Note that condition \eqref{x max<infty} of Proposition~\ref{prop:} is actually implied by its condition \eqref{p max>0}, since the r.v.\ $X$ was assumed to be real-valued; however, it will be convenient to present condition \eqref{x max<infty} explicitly. 
It follows from Proposition~\ref{prop:} that inequality \eqref{eq:ge,<e^-u} holds for all $u\in\R$ whenever $x_{\max}$ is not an atom of the distribution of the r.v.\ $X$; in particular, that will be the case 
if $\supp X$ is not bounded from above. 

One can now state

\begin{corollary}\label{cor:} 
Let $X_1,\dots,X_n$ be any r.v.'s such that the condition \eqref{eq:<infty} is satisfied for all $j\in\{1,\dots,n\}$ with $X_j$ in place of $X$. 
Then for all $u\in\R$  
\begin{equation}\label{eq:cor}
	\P\!\big(X_1+\dots+X_n>\li{L_{X_1}}(u)+\dots+\li{L_{X_n}}(u)\big)\le e^{-u}.  
\end{equation}
\end{corollary}
This follows immediately from \eqref{eq:>,<e^-u} and Theorem~\ref{th:1}. Indeed, by H\"older's inequality, 
\begin{equation}\label{eq:H}
	L_{X_1+\dots+X_n}\le L_{X_1}\fJ\cdots\fJ L_{X_n}    
\end{equation}
and hence, by \eqref{eq:L^*:=} and \eqref{eq:li L}, 
$
	\li{\big(L_{X_1+\dots+X_n}\big)}\le\li{\big(L_{X_1}\fJ\cdots\fJ L_{X_n}\big)}.   
$ 
Now Theorem~\ref{th:1} yields 
\begin{equation}\label{eq:rio}
	\li{\big(L_{X_1+\dots+X_n}\big)}(u)\le\li{L_{X_1}}(u)+\dots+\li{L_{X_n}}(u)   
\end{equation}
for all $u\in\R$. It remains to use \eqref{eq:>,<e^-u} (with $X=X_1+\dots+X_n$). 

The use of H\"older's inequality to obtain \eqref{eq:H} should explain the term ``H\"older convolution'', used in this paper for the 
operation $\fJ$. 
This operation was implicitly used (for $n=2$) by Rio in his paper \cite{rio}, which originally inspired the present study. 
However, reasoning somewhat similar in spirit to that in the proof of Theorem~\ref{th:1} was used earlier in the proofs in \cite{tashkent} and \cite[Corollary~1]{pin-utev-exp}; cf.\ also \cite[Propositions~3.1 and 3.8]{pin-hoeff}. 

In the case when $n=2$ and $X_1$ and $X_2$ are centered non-degenerate r.v.'s with moment generating functions (m.g.f.'s) finite in a neighborhood of $0$, Rio~\cite[Lemma~2.1]{rio} proved \eqref{eq:rio} for $u>0$, 
which of course implies \eqref{eq:cor} in this case. 
The assumptions that $X_1$ and $X_2$ be centered and non-degenerate r.v.'s with mg.f.'s finite in a left neighborhood of $0$ were removed in \cite{rio.transl}, and the proof was significantly shortened; in fact, display \eqref{eq:rio12} follows largely the lines of reasoning presented in \cite{rio.transl}. 
Note also that \cite[Lemma~2.1]{rio} was used in \cite{rio13}. 



\appendix
\section{Supplements and auxiliaries}\label{appendix}
First here, let us list some general properties of the functions $L^*$ and $\li L$: 


\begin{proposition}\label{prop:gen}\ 
\begin{enumerate}[(a)]
	\item \label{>-infty} 
$L^*(x)>-\infty$ for all $x\in\R$. 
	\item \label{to infty} 	$L^*(x)\to\infty$ as $x\to\infty$. 
\item \label{L^* gen} 
	The function $L^*$ is nondecreasing, convex, 
lower semi-continuous on $\R$, and hence continuous on 
its effective domain \cite{rocka} $\dom(L^*):=\{x\in\R\colon L^*(x)<\infty\}$. 
Let 
\begin{equation}\label{eq:x infty}
	x_\infty:=\sup\dom(L^*),
\end{equation}
using the standard convention $\sup\emptyset:=-\infty$; so, 
\begin{equation}\label{eq:cases}
	\mbox{(i)}\ \dom(L^*)=(-\infty,x_\infty)\quad\text{or}\quad 
	\mbox{(ii)}\ \dom(L^*)=(-\infty,x_\infty]. 
\end{equation} 
	\item \label{u pm infty}
The monotonicity of $L^*$ makes possible the definitions  
\begin{equation}\label{eq:u pm infty}
	u_{-\infty}:=\lim_{x\to-\infty}L^*(x)\quad\text{and}\quad  u_\infty:=\lim_{x\uparrow x_\infty}L^*(x)   
\end{equation}
if $x_\infty>-\infty$; 
if $x_\infty=-\infty$ (i.e., if $\dom(L^*)=\emptyset$), let $u_\infty:=\infty$; thus, in any case $-\infty\le u_{-\infty}\le u_\infty\le\infty$. 
Moreover, 
\begin{equation}\label{eq:u_infty=}
u_\infty=
\left\{	
\begin{alignedat}{2}
	&\infty&\text{ if }\dom(L^*)&=(-\infty,x_\infty), \\ 
	&L^*(x_\infty)&\text{ if }\dom(L^*)&=(-\infty,x_\infty].  
\end{alignedat}
\right. 
\end{equation}
	\item \label{li<infty} $\li L(u)\le\tli L(u)<\infty$ for all $u\in\R$. 
	\item \label{li L} 
	The functions $\li L$ and $\tli L$ are nondecreasing on $\R$ and strictly increasing on the interval $(u_{-\infty},u_\infty)$ (which may be empty). 
\end{enumerate}
\end{proposition}

\begin{proof}[Proof of Proposition~\ref{prop:gen}]
Properties \eqref{>-infty} and \eqref{to infty} in Proposition~\ref{prop:gen} follow by \eqref{eq:ne infty}. 

Properties \eqref{L^* gen} are due to the fact that, by the definition \eqref{eq:L^*:=}, $L^*$ is the pointwise 
supremum of a family of increasing affine functions on $\R$. 

The inequality $u_{-\infty}\le u_\infty$ in part \eqref{u pm infty} 
follows from the monotonicity of $L^*$, stated in part \eqref{L^* gen}, the first line in \eqref{eq:u_infty=} follows  by the lower semi-continuity of $L^*$ and \eqref{to infty}, and the second line in \eqref{eq:u_infty=} follows by \eqref{eq:u pm infty} and the continuity of $L^*$ on $\dom(L^*)$. 

The first inequality in property \eqref{li<infty} follows straight from the definitions \eqref{eq:li L} and \eqref{eq:tli L} of $\li L$ and $\tli L$, and the second inequality there follows from property \eqref{to infty}. 

Finally, concerning properties \eqref{li L}: that the functions $\li L$ and $\tli L$ are nondecreasing on $\R$ follows trivially from the definitions, and the strict increase of these functions on $(u_{-\infty},u_\infty)$ follows 
because $L^*$ is nondecreasing on $\R$ and continuous on $(-\infty,x_\infty)$, and $L^*$ maps $(-\infty,x_\infty)$ onto an interval with the endpoints $u_{-\infty}$ and $u_\infty$. 	
\end{proof}

The following proposition deals with the important special case when $L$ is convex and, at least partially, strictly convex. 


\begin{proposition}\label{prop:str}\ 
Suppose that the function $L$ is convex on $(0,\infty)$ and strictly convex (and hence real-valued) on an interval $(t_0,t_1)$ such that $0<t_0<t_1<\infty$. 
For any $t\in(t_0,t_1)$, let $L'(t)$ and $L'(t-)$ denote, respectively, the right and left derivatives of $L$ at $t$. 
Let also $L'(t_0+):=\lim_{t\downarrow t_0}L'(t)$. 
Then the interval $\big(L'(t_0+),L'(t_1-)\big)$ is nonempty and the function $L^*$ is strictly increasing 
on $\big(L'(t_0+),L'(t_1-)\big)$.  
Moreover, $\big(L'(t_0+),L'(t_1-)\big)
\subseteq\dom(L^*)$ and $L^*$ is strictly increasing on the entire interval $\big(L'(t_0+),\infty\big)\cap\dom(L^*)$. 
\end{proposition}

\begin{proof}[Proof of Proposition~\ref{prop:str}]
The strict convexity of $L$ implies that $L'$ is strictly increasing and right-continuous on the nonempty interval $(t_0,t_1)$. So, the interval $\big(L'(t_0+),L'(t_1-)\big)$ is nonempty and for any 
$x$ in this interval one has 
\begin{align*}
	t_x:={L'}^{-1}(x)
	&:=\inf\{t\in(t_0,t_1)\colon L'(t)\ge x\}
	=\min\{t\in(t_0,t_1)\colon L'(t)\ge x\}\in(t_0,t_1). 
\end{align*}
Clearly, $t_x$ is nondecreasing in $x\in\big(L'(t_0+),L'(t_1-)\big)$, and  
\begin{equation}\label{eq:< <}
	L'(t_x-)\le x\le L'(t_x) 
\end{equation}
for all $x\in\big(L'(t_0+),L'(t_1-)\big)$.

Take any $x\in\big(L'(t_0+),L'(t_1-)\big)$. 
The right derivative of $tx-L(t)$ in $t\in(t_0,t_1)$ is $x-L'(t)$, which is greater than $0$ for $t\in(0,t_x)$ and no greater than $0$ for $t\in[t_x,\infty)$. So, 
$tx-L(t)$ is increasing in $t\in(t_0,t_x]$ and non-increasing in $t\in[t_x,t_1)$. 
Since $tx-L(t)$ is concave in $t\in(0,\infty)$, it follows that $tx-L(t)$ is increasing in $t\in(0,t_x]$ and non-increasing in $t\in[t_x,\infty)$. Thus, 
\begin{equation}\label{eq:=}
	L^*(x)=t_x x-L(t_x),  
\end{equation}
for all $x\in\big(L'(t_0+),L'(t_1-)\big)$. 

Take now any $x$ and $y$ such that $L'(t_0+)<x<y<L'(t_1-)$. For brevity, let here $s:=t_x$ and $t:=t_y$, so that one has $0<t_0<s\le t<t_1$. If $s=t$ then $L^*(x)=sx-L(s)=tx-L(t)<ty-L(t)=L^*(y)$, which yields $L^*(x)<L^*(y)$. 
In the remaining case, when $s<t$, recall \eqref{eq:< <} and write 
\begin{align*}
	L^*(y)-L^*(x)&=[ty-L(t)]-[sx-L(s)] \\ 
	&\ge[tL'(t-)-L(t)]-[sL'(s)-L(s)] \\
	&=s[L'(t-)-L'(s)]+\int_s^t[L'(t-)-L'(\tau)]\dd\tau>0,  
\end{align*}
since $s>0$ and $L'$ is strictly increasing on $(t_0,t_1)$. 
Thus, in either case $L^*(x)<L^*(y)$. 

This shows that indeed $L^*$ is strictly increasing 
on the open interval $\big(L'(t_0+),L'(t_1-)\big)$, which in turn immediately implies that $\big(L'(t_0+),L'(t_1-)\big)
\subseteq\dom(L^*)$. 
The final sentence of Proposition~\ref{prop:str} follows because of this general fact \big(implicitly used to obtain \eqref{eq:=}\big): if a function $f$ is convex on $(0,\infty)$, $0\le a<b\le c<\infty$, $(a,c]\subseteq\dom f$, and $f$ is strictly increasing on $(a,b)$, then $f$ is so on the entire interval $(a,c]$. 
\end{proof}

The following lemma seems to be of general interest; it will also be used in the proof of Proposition~\ref{prop:}. 


\begin{lemma}\label{lem:} 
For brevity, let here $L:=L_X$ and recall the definitions of $x_\infty$ and $u_\infty$ in \eqref{eq:x infty} and \eqref{eq:u pm infty}. 
Then 
\begin{equation}\label{eq:x=x}
	x_\infty=x_{\max}.   
\end{equation}
Moreover, if condition \eqref{x max<infty} 
of Proposition~\ref{prop:} holds, then 
\begin{equation}\label{eq:=-ln p_max}
	u_\infty=L^*(x_{\max})=-\ln p_{\max};       
\end{equation}
if at that condition \eqref{p max>0} 
of Proposition~\ref{prop:} holds as well, then 
\begin{equation}\label{eq:li L=}
	\li L(u)=
	\left\{
	\begin{aligned}
	x_{\max} &\text{ if } u\in(-\ln p_{\max},\infty), \\ 
	x_{\max} &\text{ if } u=-\ln p_{\max} \text{ and } p_{\max}<1, \\ 
	-\infty &\text{ if } u=-\ln p_{\max} \text{ and } p_{\max}=1.    
	\end{aligned}
	\right.
\end{equation}
\end{lemma}

\begin{proof}[Proof of Lemma~\ref{lem:}]
If for some real $y$ one has $y\in(x_{\max},\infty)$, then $x_{\max}\in\R$ and $\P(X\le x_{\max})=1$, whence $L(t)\le tx_{\max}$ for all $t>0$ and $L^*(y)\ge\sup_{t>0}[ty-tx_{\max}]=\infty$. 
On the other hand, if $y\in(-\infty,x_{\max})$ then $p_y:=\P(X>y)>0$ 
and hence $L(t)\ge ty+\ln p_y$ for all $t>0$ and $L^*(y)\le
-\ln p_y<\infty$. 
Thus, $L^*(y)=\infty$ for all $y\in(x_{\max},\infty)$ and $L^*(y)<\infty$ for all $y\in(-\infty,x_{\max})$. Now \eqref{eq:x=x} follows. 

To prove \eqref{eq:=-ln p_max}, suppose that indeed $x_{\max}<\infty$; then, by the definition of $x_{\max}$ in condition \eqref{x max<infty} 
of Proposition~\ref{prop:}, $x_{\max}\in\R$. 
The first equality in \eqref{eq:=-ln p_max} now follows by \eqref{eq:u_infty=} and \eqref{eq:x=x}; 
indeed, \eqref{eq:u_infty=} shows that $u_\infty=L^*(x_\infty)$ whenever $x_\infty\in\R$. 
Concerning the second equality in \eqref{eq:=-ln p_max}, 
one has $\P(X\le x_{\max})=1$ and hence $\E\exp\{t(X-x_{\max})\}$ is non-increasing in $t$. Using this monotonicity and dominated convergence, one sees that   
\begin{equation*}
	\exp\{-L^*(x_{\max})\}=\inf_{t>0}\E\exp\{t(X-x_{\max})\}
	=\lim_{t\to\infty}\E\exp\{t(X-x_{\max})\}=\E\ii\{X=x_{\max}\}=p_{\max}; 
\end{equation*}
here $\exp\{-L^*(x_{\max})\}$ is naturally understood as $0$ when $L^*(x_{\max})=\infty$. 
Thus, \eqref{eq:=-ln p_max} is completely proved.  

It remains to prove \eqref{eq:li L=}, assuming that $x_{\max}<\infty$ and $p_{\max}>0$, 
so that, by \eqref{eq:=-ln p_max}, $u_\infty<\infty$. 
Now \eqref{eq:u_infty=} and \eqref{eq:cases} yield $\dom(L^*)=(-\infty,x_\infty]$, whence, by 
\eqref{eq:x=x},   
%
\begin{equation}\label{eq:dom=}
	\dom(L^*)=(-\infty,x_{\max}]. 
\end{equation}
Take any $u\in\R$. 

If $u\in(-\ln p_{\max},\infty)$ then, by \eqref{eq:=-ln p_max} and \eqref{eq:dom=}, 
$L^*(x_{\max})<u<L^*(x)$ for all $x\in(x_{\max},\infty)$ and hence 
%
$\li L(u)=
x_{\max}$. 

Consider next the case when $u=-\ln p_{\max}$ 
and $0<p_{\max}<1$. 
Then, in view of \eqref{eq:L_X} and \eqref{eq:<infty}, 
$L''(t)(\E e^{tX})^2=\E X^2 e^{tX}\E e^{tX}-(\E X e^{tX})^2>0$ (by the Cauchy--Schwarz inequality) for all $t$ in a right neighborhood of $0$, so that $L$ is strictly convex in that neighborhood. 
Hence, by Proposition~\ref{prop:str} and \eqref{eq:dom=}, the function $L^*$ is 
strictly increasing on a nonempty interval of the form $(x,x_{\max}]$. 
So, since the function $L^*$ is nondecreasing on $\R$, it is strictly less than $L^*(x_{\max})$ on the interval $(-\infty,x_{\max})$. 
Hence, by \eqref{eq:li L} and \eqref{eq:=-ln p_max}, indeed 
$\li L(u)=x_{\max}$ if $u=-\ln p_{\max}$ and $p_{\max}<1$. 

Finally, if $p_{\max}=1$ \big(i.e., $\P(X=x_{\max})=1$\big), then $L(t)=tx_{\max}$ for all $t\in(0,\infty)$, $L^*(x)$ equals $0$ for $x\in(-\infty,x_{\max}]$ and $\infty$ for $x\in(x_{\max},\infty)$, so that 
$\li L(u)$ equals $-\infty$ for $u\in(-\infty,0]$ and $x_{\max}$ for $u\in(0,\infty)$. 
In particular, $\li L(u)=-\infty$ if $u=0=-\ln p_{\max}$. 
\end{proof}


Now one is ready for 

\begin{proof}[Proof of Proposition~\ref{prop:}]
Let again, for brevity, $L:=L_X$. 

To prove the the ``only if'' part of Proposition~\ref{prop:}, take any $u\in\R$ violating \eqref{eq:ge,<e^-u}. Then, by \eqref{eq:>,<e^-u}, 
\begin{equation*}
	\P\!\big(X=\li L(u)\big)>0. 
\end{equation*}

Therefore and because 
the r.v.\ $X$ was assumed to be real-valued,  
necessarily 
\begin{equation}\label{eq:li L(u)>-infty}
\li L(u)>-\infty, 	
\end{equation}
which, by the definition of $u_{-\infty}$ in \eqref{eq:u pm infty}, in turn implies $u>u_{-\infty}$, so that the interval $(u_{-\infty},u]$ is nonempty. 

If now $u<u_\infty$ then the interval $(u_{-\infty},u]$ is 
contained in the interval $(u_{-\infty},u_\infty)$, on which, by Proposition~\ref{prop:gen}\eqref{li L}, the function $\li L$ is strictly increasing. So, for any 
$v\in(u_{-\infty},u)$ one has $\li L(u)>\li L(v)$ and hence  
\begin{equation*}
	\P\!\big(X\ge\li L(u)\big)\le \P\!\big(X>\li L(v)\big)\le e^{-v},   
\end{equation*} 
by \eqref{eq:>,<e^-u} with $v$ in place of $u$; 
letting now $v\uparrow u$, one obtains \eqref{eq:>,<e^-u}, which contradicts the assumption on $u$. 

Thus, it is necessary that $u\in[u_\infty,\infty)$. 
Then, of course, $u_\infty<\infty$, which, in view of \eqref{eq:u_infty=},  
means that case (ii) in \eqref{eq:cases} takes place, whence  
necessarily $x_\infty\in\R$. 
So, by \eqref{eq:x=x}, condition \eqref{x max<infty} of Proposition~\ref{prop:} holds. 
Further, by \eqref{eq:=-ln p_max}, $u_\infty=-\ln p_{\max}$ and hence $u\in[-\ln p_{\max},\infty)$. 

If $u=-\ln p_{\max}$ and $p_{\max}<1$ then, by \eqref{eq:li L=}, $\li{L}(u)=x_{\max}$ and hence  $\P\!\big(X\ge\li{L}(u)\big)=\P\!\big(X\ge x_{\max}\big)
=\P\!\big(X=x_{\max}\big)=p_{\max}=e^{-u}$, so that \eqref{eq:ge,<e^-u} holds. 
If $u=-\ln p_{\max}$ and $p_{\max}=1$ then, by \eqref{eq:li L=}, $\li{L}(u)=-\infty$, 
which contradicts \eqref{eq:li L(u)>-infty}. 

We conclude that any $u\in\R$ violating \eqref{eq:ge,<e^-u} must satisfy condition \eqref{u ge} of Proposition~\ref{prop:} as well. 
Thus, the ``only if'' part is proved.


The ``if'' part of Proposition~\ref{prop:} follows immediately by Lemma~\ref{lem:}. Indeed, if $p_{\max}:=\P(X=x_{\max})>0$ and $u>-\ln p_{\max}$, then the right-hand side of \eqref{eq:ge,<e^-u} is strictly less  $p_{\max}=\P\!\big(X=x_{\max}\big)$, whereas, by \eqref{eq:li L=}, the left-hand side of \eqref{eq:ge,<e^-u} 
equals $\P\!\big(X\ge x_{\max}\big)$, which 
is no less than $\P\!\big(X=x_{\max}\big)$. 
Proposition~\ref{prop:} is now completely proved. 
\end{proof}

\textbf{Acknowledgment.}\ I am pleased to thank Emmanuel Rio for bringing his work \cite{rio00,rio12,rio.transl,rio13} to my attention. 

\bibliographystyle{abbrv}


\bibliography{C:/Users/Iosif/Dropbox/mtu/bib_files/citations12.13.12}
\end{document}